\documentclass[12pt]{amsart}
\usepackage{amsmath,amsthm,amsfonts,amssymb,eucal}
\renewcommand {\a}{ \alpha }
\renewcommand{\b}{\beta}

\newcommand{\g}{\gamma}

\newcommand{\varf}{\varphi}
\renewcommand{\d}{\delta}

\newcommand{\D}{\Delta}

\newcommand{\Sg}{\Sigma}
\renewcommand{\l}{\lambda}

\newcommand{\p}{\partial}
\newcommand{\om}{\omega}
\newcommand{\Om}{\Omega}
\newcommand{\R}{ \mathbb R}

\newcommand{\N}{ \mathbb N}

\newcommand{\Z}{ \mathbb Z}

\newcommand{\CD}{\mathcal D}

\newcommand{\CF}{\mathcal F}

\newcommand{\CH}{\mathcal H}
\newcommand{\CI}{\mathcal I}

\newcommand{\CR}{\mathcal R}
\newcommand{\CS}{\mathcal S}

\newcommand{\CX}{\mathcal X}

\newcommand{\CN}{\mathcal N}

\newcommand {\GS}{\mathfrak S}

\newcommand {\bh}{\mathbf h}

\newcommand {\bm}{\mathbf m}
\newcommand {\bt}{\mathbf t}
\newcommand {\bx}{\mathbf x}
\newcommand {\by}{\mathbf y}

\newcommand {\BA}{\mathbf A}

\newcommand {\BH}{\mathbf H}

\newcommand {\BM}{\mathbf M}

\newcommand {\BQ}{\mathbf Q}

\newcommand {\BT}{\mathbf T}

\newcommand{\gz}{\mathfrak z}

\newcommand{\wh}{\widehat}

\DeclareMathOperator{\meas}{meas}

\newtheorem{thm}{Theorem}[section]

\newtheorem{lem}[thm]{Lemma}
\newtheorem{prop}[thm]{Proposition}

\theoremstyle{definition}
\newtheorem{example}[thm]{Example}
\theoremstyle{remark}
\newtheorem{rem}[thm]{Remark}

\numberwithin{equation}{section}

\newcommand{\thmref}[1]{Theorem~\ref{#1}}

\newcommand{\lemref}[1]{Lemma~\ref{#1}}
\newcommand{\bsymb}{\boldsymbol}

\newcommand{\sh}{Schr\"odinger }
\newcommand{\vs}{\vskip0.2cm}

\newcommand{\loc}{\rm{loc}}
\newcommand{\w}{\infty}

\newcommand{\Tr}{\rm{Tr}}

\begin{document}

\title[Spectral problems on the half-line]{On a class of spectral problems on the half-line and their applications to multi-dimensional problems}

\author[M. Solomyak]{M. Solomyak}
\address{Department of Mathematics
\\ Weizmann Institute\\ Rehovot\\ Israel}
\email{michail.solomyak@weizmann.ac.il}


\begin{abstract}
A survey of estimates on the number $N_-(\BM_{\a G})$ of negative eigenvalues (bound states) of the Sturm-Liouville operator
$\BM_{\a G}u=-u''-\a G$ on the half-line, as depending on the properties of the function $G$ and the value of the coupling parameter $\a>0$. The central result is \thmref{S1/2a} giving a sharp sufficient condition for the semi-classical behavior $N_-(\BM_{\a G})=O(\a^{1/2})$, and the necessary and  sufficient conditions for a "super-classical" growth rate
$N_-(\BM_{\a G})=O(\a^q)$ with any given $q>1/2$. Similar results for the problem on the whole $\R$ are also presented. Applications to the multi-dimensional spectral problems are discussed.
\end{abstract}

\maketitle

\section{Introduction}\label{intro}
\subsection{Preliminary remarks}\label{prel}
The main goal of this paper is to give a comprehensive survey of the results on the estimates of the number of the negative eigenvalues (bound states) of the Sturm-Liouville operator on the half-line $\R_+$, of the form
\begin{equation*}
    (\BM_Gu)(t)=-u''(t)-G(t)u(t),\ t>0;\qquad u(0)=0.
\end{equation*}
The function $G\ge0$ is always supposed to be integrable over any interval $(a,b),\ 0<a<b<\w$. Further conditions on $G$ will be imposed later. The operator $\BM_G$ acts in the Hilbert space $L_2(\R_+)$ and is defined via its quadratic form,
\begin{equation}\label{qf}
    \bm_G[u]=\int_0^\w(|u'|^2-G|u|^2)dt,\qquad u\in H^{1,0}(\R_+)=\{u\in H^1(\R_+):u(0)=0\}.
\end{equation}

The conditions guaranteeing that this quadratic form is bounded from below and closed, are well known, and we recall them in the beginning of Section \ref{bbeh}.

Such problems are of their own interest. Besides, they frequently arise in the spectral analysis of some multi-dimensional problems, where they may strongly affect the behavior of the spectrum. We illustrate it by some examples in Section \ref{prim}.

Given a self-adjoint, bounded from below operator $\BA$ whose negative spectrum is discrete,
we denote by $N_-(\BA)$ the number of its negative eigenvalues, counted with multiplicities.
So, we study the quantity $N_-(\BM_G)$ as depending on the properties of $G$.

Historically, the first result in this direction was the famous Bargmann estimate \cite{Barg}:
\begin{equation}\label{bargest}
    N_-(\BM_G)\le\int_0^\w tG(t)dt.
\end{equation}
The next result, usually called Calogero estimate, was obtained simultaneously and independently in the papers \cite{calog,cohn}. It says that if the potential $G$ is monotone, then
\begin{equation}\label{calest}
    N_-(\BM_G)\le\frac2{\pi}\int_0^\w\sqrt{G(t)}dt.
\end{equation}

As a rule, one considers this problem not for an "individual" potential $G$, but rather for the family
$\a G$, where $\a>0$ is a large parameter (the coupling constant). The corresponding family of operators is given by
\begin{equation}\label{poluos-a}
    (\BM_{\a G}u)(t)=-u''(t)-\a G(t)u(t),\ t>0;\qquad u(0)=0.
\end{equation}
Here one is interested in the behavior of the function
$N_-(\BM_{\a G})$ as $\a\to\w$. The "semi-classical"
behavior
\begin{equation}\label{kvazi}
    N_-(\BM_{\a G})=O(\a^{1/2})
\end{equation}
is typical, it occurs if the potential $G$ decays
fast enough. For slowly decaying potentials a "super-classical" behavior
\begin{equation}\label{cuper}
    N_-(\BM_{\a G})=O(\a^q),\qquad q>1/2
\end{equation}
(as well as non-powerlike behavior) is also possible.\vs

The further results discussed in this paper are mostly due to
M.Sh. Birman, to the author of the present paper, and to our
students and colleagues. The most important facts are
collected in \thmref{S1/2a} that gives a sharp (close to the
necessary) condition on $G$ for the semi-classical behavior
\eqref{kvazi}, and also the necessary AND sufficient
condition for the behavior as in \eqref{cuper}, with any
given $q>1/2$. Usually, these estimates were being obtained as
a "by-product", in the course of the analysis of various multi-dimensional
problems of a similar nature. For this reason, the results
for the operators \eqref{poluos-a} were often published
in an insufficiently complete form, and they did not draw
enough attention of the specialists in the field. That is why
the author considers it useful to collect the basic facts in one short
survey paper. The central results are given with proofs.

We also describe several typical multi-dimensional
problems, whose spectral analysis uses these results  as an important ingredient. We would like to stress the special role of the super-classical estimates in this analysis.\vs

Substituting in \eqref{bargest} $\a G$ for $G$, we see that
the Bargmann estimate gives an excessive rate of growth
\eqref{cuper}, with $q=1$. The estimate \eqref{calest}
gives the semi-classical growth \eqref{kvazi}, however, only for the monotone potentials. For non-monotone potentials the
problem remained open for a long period of time.\vs

Starting from the seminal papers by
Birman \cite{Bir} and by Schwinger \cite{Schw}, it became standard to consider, along with the family $\BM_{\a G}$, the corresponding {\it Birman -- Schwinger operator} $\BT_G$. This operator acts in the "homogeneous" Sobolev space
\[ \CH^{1,0}(\R_+)=\{u\in H^1_{\loc}(\R_+):u(0)=0;\ u'\in L^2(\R_+)\},\]
it is generated by the quadratic form
\begin{equation}\label{btG}
    \bt_G[u]=\int_0^\w G(t)|u(t)|^2dt.
\end{equation}
For the problems we discuss in this paper, the operator
$\BT_G$ is compact. Given a compact, self-adjoint
operator $\BT\ge0$ and a number $s>0$, we write
$n_+(s,\BT)$ for the number of eigenvalues $\l_k(\BT)$ (counted
with multiplicities), such that $\l_k(\BT)>s$.
More generally, for any countable family $\CS$ of
non-negative numbers we write
\[ n_+(s,\CS)=\#\{x\in\CS:x>s\}.\]
So, $n_+(s,\BT)=n_+(s,\{\l_k(\BT)\})$.
By the classical Birman -- Schwinger principle, the equality
\begin{equation}\label{bsch}
    N_-(\BM_{\a G})=n_+(s,\BT_G),\qquad s=\a^{-1},
\end{equation}
is valid for any $\a>0$. This allows one to reduce the study of the function $N_-(M_{\a G})$ to spectral estimates for the "individual" operator $\BT_G$. Its spectrum can be calculated using the Rayleigh quotient
\begin{equation}\label{rayy}
    \CR(u)=\frac{\int_0^\w G(t)|u(t)|^2dt}{\int_0^\w|u'|^2dt},\qquad u\in\CH^{1,0}(\R_+).
\end{equation}\vs

This reduction was applied to the estimates of $N_-(\BM_{\a G})$ in the paper \cite{BirBor}, where also the multi-dimensional problems of this type were analyzed.
The next statement is a particular case (for $l=1$ and $(a,b)=\R_+$) of Lemma 3.1 in \cite{BirBor}, "translated" into the language of operators $\BM_{\a G}$.

\begin{prop}\label{Birb}
Let $G\ge 0$ be such that for some non-increasing function $\varf>0$ on $\R_+$ we have
\begin{equation*}
    R:=\int_0^\w \frac{dt}{\varf(t)}<\w,\qquad S:=\int_0^\w G(t)\varf(t)dt<\w.
\end{equation*}
Then the estimate is satisfied:
\begin{equation*}
    N_-(\BM_{\a G})\le C\a^{1/2}RS,
\end{equation*}
with some constant $C$ independent of the potential $G$.
\end{prop}

This estimate applies to arbitrary (i.e., not necessarily monotone) potentials and gives the semi-classical growth as $\a\to\w$. If $G$ is monotone, then the choice $\varf=\sqrt{G}$ gives  estimate \eqref{calest}, up to the value of the estimating constant. Note also that  the Weyl type asymptotic formula
\begin{equation}\label{weylalpha}
    N_-(\a,\BM_{\a G})\sim\pi^{-1}\a^{1/2}\int_0^\w \sqrt Gdt,\qquad \a\to \w,
\end{equation}
is valid under the assumptions of Proposition \ref{Birb}; this also was shown in \cite{BirBor}.

It turned out later that it is more convenient to use another type of estimates, expressing the conditions on $G$ in terms of the number sequence
\begin{equation}\label{posl}
    \bsymb\gz(G)=\{\gz_j(G)\}_{j\in\Z}:\qquad \gz_j(G)=2^j\int_{I_j}G(t)dt\ {\rm{where}}\ I_j=(2^{j-1},2^j).
\end{equation}
As we will see in Section \ref{mainr}, this language is adequate for the problems considered: it allows us to give conditions that guarantee the membership of $\BT_G$ in various operator classes. These conditions are sharp, and in many important cases even necessary and sufficient.

\medskip
\noindent
{\it Acknowledgements.}
The author expresses his gratitude to G. Rozenblum for useful suggestions and
valuable discussion.

\section{Classes of number sequences and classes of operators}\label{prostr}
The conditions on the sequence $\bsymb\gz(G)$ will be formulated in terms of the spaces $\ell_q$ and their weak analogs $\ell_{q,\w}$,  and the results on the
operator $\BT_G$ will be formulated in terms of the Neumann-Schatten ideals $\GS_q$ and their weak analogs that we denote by $\Sg_q$.

Given a number sequence $\bx=\{x_k\}$, we write $|\bx|=\{|x_k\}$.
A sequence
$\bx$ lies in $\ell_{q,\w},\ q>0,$ if and only if
\begin{equation}\label{kvaz}
    \|\bx\|_{q,\w}^q:=\sup\limits_{s>0}s^q n_+(s,|\bx|)<\w.
\end{equation}
The functional $\|\cdot\|_{q,w}$ is a quasi-norm in
$\ell_{q,\w}$. This means that instead of the standard
triangle inequality, a weaker property
\[ \|\bx+\by\|_{q,w}\le c(q)(\|\bx\|_{q,w}+\|\by\|_{q,w})\]
is satisfied, with some constant $c(q)>1$ that does not depend on $\bx,\by$. The quasi-norm \eqref{kvaz} generates a topology on $\ell_{q,\w}$, in which this space is non-separable. The closure of the set of elements $\bx$ with only a finite number of non-zero terms is a separable subspace in $\ell_{q,\w}$. It is denoted by
$\ell_{q,\w}^\circ$.

If $q>1$, and only in this case,
there exists a norm on $\ell_{q,\w}$, equivalent to the above quasi-norm. However, even when $q>1$, we will use the quasi-norm $\|\cdot\|_{q,\w}$ for the estimates. \vs
Let us also recall that the spaces
$\ell_q$ with $q<1$ are quasi-normed with respect to the standard quasi-norm
\begin{equation*}
    \|\{x_j\}\|_q=(\sum_j|x_j|^q)^{1/q}.
\end{equation*}
These spaces are non-normalizable. It is clear that $\ell_q\subset\ell_{q,\w}^\circ$ for any $q>0$
\vs

By definition, the space $\GS_q,\ 0<q<\w,$ is formed by the compact operators $\BT$, whose sequence of singular numbers $\{s_k(\BT)\}$ lies in $\ell_q$.  The spaces $\Sg_q$ and $\Sg_q^\circ$ are formed by the compact operators, such that this sequence lies in $\ell_{q,\w}$, or in $\ell_{q,\w}^\circ$ respectively. The (quasi-)norms in these spaces are induced by this definition. Recall that $\GS_\w$ standardly denotes the space of all compact operators.

For any sequence $\bx\in\ell_{q,\w}$ of non-negative numbers we define the (non-linear) functionals
\begin{equation}\label{Dd}
    \D_q(\bx)=\limsup\limits_{s\to0}s^q
n_+(s,\bx),\qquad \d_q(\bx)=\liminf\limits_{s\to0}s^q
n_+(s,\bx).
\end{equation}
It is clear that $\d_q(\bx)\le\D_q(\bx)\le\|\bx\|_{q,\w}^q$ and $\ell_{q,\w}^\circ=\{\bx\in\ell_{q,\w}:\D_q(\bx)=0\}.$

For a non-negative operator $\BT\in\Sg_q$ we define
\begin{equation}\label{Ddop}
    \D_q(\BT)=\D_q(\{s_k(\BT)\}),\qquad  \d_q(\BT)=\d_q(\{s_k(\BT)\}).
\end{equation}
See the book \cite{BSbook} for more detail about the spaces $\GS_q,\Sg_{q,\w},\Sg_{q,\w}^\circ$, including the case $q\le 1$.

\section{Main results on the operator $\BT_G$}\label{mainr}
Here we formulate our main results. Their proofs, or the necessary references, will be given in the next section.  Given a potential $G$, we
define the sequence $\bsymb\gz(G)$  as in \eqref{posl}, and the operator $\BT_G$ in the space $\CH^{1,0}(\R_+)$, associated with the quadratic form \eqref{btG} (or, equivalently, with the Rayleigh quotient \eqref{rayy}).

Our first result is rather elementary. It shows that the conditions of the boundedness and of the compactness of the operator $\BT_G$ can be conveniently expressed in terms of the sequence $\bsymb\gz(G)$.

\begin{thm}\label{ogr}
The operator $\BT_G$ is bounded if and only if $\bsymb\gz(G)\in\ell_\w(\Z)$, and
\begin{equation}\label{TGogr1}
   \frac12 \|\bsymb\gz(G)\|_\w\le\|\BT_G\|\le 8\|\bsymb\gz(G)\|_\w.
\end{equation}

The operator $\BT_G$ is compact if and only if $\gz_j(G)\to0$ as $j\to\pm\w$.
\end{thm}
In Subsection \ref{dok-ogr} we derive it from a well-known criterion due to Hille.
\vs

The next theorem gives a simple, but quite useful lower estimate for the function $n_+(s,\BT_G)$. Its proof is given in Subsection \ref{dok-below}.
\begin{thm}\label{estbelow}
Let $G\ge0$ be a function on $\R_+$, such that the operator $\BT_G$ is compact. Then for any $s>0$ the estimate is satisfied,
\begin{equation}\label{lowestim}
    2n_+(s,\BT_G)\ge n_+(\g s,\bsymb\gz(G))
\end{equation}
where $\g>0$ is an absolute constant.
\end{thm}

The following result gives a general sufficient condition for the inclusion $\BT_G\in\Sg_{1/2}$ and for
the Weyl type asymptotic behavior of the function $n_+(s,\BT_G)$. This corresponds to the semi-classical behavior of the function $N_-(\BM_{\a G})$ in the large coupling constant regime. The result goes back to the lectures \cite{BSlect}, see \S 4.8 there.
Taking into account its importance, we present the proof (in Subsection \ref{dok-1/2}).
\begin{thm}\label{S1/2}
Suppose $\bsymb\gz(G)\in\ell_{1/2}$. Then $\BT_G\in\Sg_{1/2}$, and there exists a constant $C>0$, such that
\begin{equation}\label{basicest}
    \|\BT_G\|_{1/2,\w}\le C\|\bsymb\gz(G)\|_{1/2},
\end{equation}
or, equivalently,
\begin{equation}\label{basicest1}
   n_+(s,\BT_G)\le Cs^{-1/2}\sum_{j\in\Z}\gz^{1/2}_j(G).
\end{equation}

Under this assumption the Weyl type asymptotic formula holds:
\begin{equation}\label{weyl}
    n_+(s,\BT_G)\sim\pi^{-1}s^{-1/2}\int_0^\w \sqrt Gdt,
    \qquad s\to 0.
\end{equation}
\end{thm}

\thmref{S1/2} is slightly stronger a result compared with Proposition \ref{Birb}; this was proven in the paper \cite{BLS}, whose main purpose was to extend the results, already known for the operators of the type
$\BT_G$, to higher order operators, and to
similar problems on vector-valued functions.

\thmref{S1/2} gives a very convenient, but still only sufficient condition for $\BT_G\in\Sg_{1/2}$. The necessary condition, which is
\[\BT_G\in\Sg_{1/2}\ \Longrightarrow \bsymb\gz(G)\in\ell_{1/2,\w},\]
immediately follows from \thmref{estbelow}.

The necessary and sufficient condition for  $\BT_G\in\Sg_{1/2}$ is also known,
it was obtained in \cite{NS}, Theorem C. We present it below, without proof. To formulate it, we need one more notation. For any finite interval $I\subset\R_+$ and a non-negative function $G$ on $I$, we denote by $\BT_{G,I}$ the self-adjoint operator in $H^{1,0}(I)$, whose corresponding Rayleigh quotient is similar to \eqref{rayy}, but with the integration over $I$, on the domain $H^{1,0}(I)$.
It is well-known that the function $n_+(s,\BT_{G,I})$ obeys the Weyl asymptotic law,
and hence, $n_+(s,\BT_{G,I})=O(s^{-1/2})$ as $s\to 0$.
The assumption \eqref{dopusl} below requires such estimate for the direct sum of Dirichlet problems on all intervals $I_j$.
\begin{thm}\label{N-S}
The two conditions:
\[\bsymb\gz(G)\in\ell_{1/2,\w}\]
and
\begin{equation}\label{dopusl}
    \sup\limits_{s>0}\sum_{j\in\Z}s^{1/2}n_+(s,\BT_{G,I_j})
    <\w,\qquad I_j=(2^{j-1},2^j)
\end{equation}
are necessary and sufficient for $\BT_G\in\Sg_{1/2}$.
\end{thm}

There are many ways to check the condition \eqref{dopusl} in concrete situations. However, it cannot be expressed in terms of the sequence $\bsymb\gz(G)$ alone.

The necessary and sufficient condition guaranteeing
the validity of the asymptotic formula \eqref{weyl} was also established in \cite{NS}, Theorem D. Its formulation is more involved, and we do not present it in this paper.

Based upon these two results, several examples of the potential $G$ were constructed in \cite{NS}, such that $n_+(\BT_G)\sim cs^{-1/2}$, with some constant $c$ that differs
from the one appearing in \eqref{weyl}. \vs

  The last result in this section gives the necessary and sufficient conditions for the operator $\BT_G$ to lie in the spaces, intermediate between $\Sg_{1/2}$ and
 $\GS_\w$. It was obtained in \cite{BS-91}, Section 6, and the detailed exposition was presented in \cite{BL}. The proof is outlined in Subsection \ref{proof-int}.

\begin{thm}\label{interp}
The operator $\BT_G$ belongs to $\Sg_q, \Sg_q^\circ$, or to $\GS_q$ with some $q\in(1/2,\w)$
if and only if the sequence $\bsymb\gz(G)$ lies in the corresponding class $\ell_{q,\w}, \ell_{q,\w}^\circ$, or $\ell_q$ respectively. The two-sided estimates are satisfied:
\begin{equation}\label{est1}
    c_q\|\bsymb\gz(G)\|_{q,\w}\le\|\BT_G\|_{q,\w}\le C_q\|\bsymb\gz(G)\|_{q,\w},
\end{equation}
\begin{equation}\label{est2}
    c_q\D_q(\bsymb\gz(G))\le \D_q(\BT_G)\le C_q\D_q(\bsymb\gz(G)),
\end{equation}
\begin{equation}\label{est3}
    c_q\|\bsymb\gz(G)\|_q\le\|\BT_G\|_q\le C_q\|\bsymb\gz(G)\|_q.
\end{equation}
\end{thm}
\vs
\begin{rem}\label{eta} Sometimes, one defines the sequence $\{\gz_j(G)\}$
by the equality $\gz_j(G)=\int_{I_j}sG(s)ds$. It is also possible to divide $\R_+$ into the family of intervals
$(c^{j-1}, c^j)$ with an arbitrary $c>1$. It is clear that all this affects only the values of the estimating constants.
\end{rem}

\begin{rem}\label{rem2}
The presence in \eqref{basicest1} of the terms with $j\to-\w$ reflects the fact that due to the
Dirichlet boundary condition at $t=0$ the admissible weight function $G$ may have a non-integrable singularity at this point. If for some additional reasons we restrict ourselves to the functions $G$ that are integrable in a vicinity of $t=0$, then in \eqref{basicest1} it is possible to replace the sum of terms with $j\le0$ by one term, $(\int_0^1 G(s)ds)^{1/2}$. This coarsens the upper estimate, but makes it look simpler. Clearly, the lower estimates in \eqref{est1} and \eqref{est3} do not survive. The lower estimate in \eqref{est2} remains valid.

This was the way to define the sequence $\bsymb{\gz(G)}$ in the lectures \cite{BSlect}, and also in the paper \cite{BL}.
In this connection, see also Subsections \ref{drugie} and \ref{pryam} of the present paper.
\end{rem}
\begin{rem}\label{rem3}
Estimate \eqref{est3} for $q=1$ says that the condition $\bsymb\gz(G)\in\ell_1$ is necessary and sufficient for
the operator $\BT_G$ to be trace class. This condition is equivalent to $\CI(G):=\int tG(t)dt<\w$. So, it yields Bargmann's estimate \eqref{bargest}, up to the constant factor.
In this connection we note that actually,
$\CI(G)= \Tr\BT_G.$
\end{rem}
\subsection{An unsolved problem.}\label{neresh}
In all the results formulated above, we assumed $G\ge 0$. The results for sign-indefinite potentials then follow from the variational principle, since
\[ N_-(\BM_{\a G})\le N_-(\BM_{\a G_+}),\qquad 2G_+=|G|+G.\]
In principle, it is possible that for such potentials the better estimates can be obtained, that take into account the interplay between the positive and the negative parts of $G$. Such estimates are known for the operator norm $\|\BT_G\|$, they
follow from the results of the paper \cite{MazVe}. Much earlier, some qualified estimates
for $N_-(\BM_{\a G}$ with sign-indefinite $G$ were obtained in \cite{cha-m}. More recently this problem was analyzed in \cite{DHS}. However, no estimates, giving for such potentials the order \eqref{kvazi}, are known
till now, and establishing them is an interesting and important problem.

\section{Proofs}\label{dok}
\subsection{Proof of \thmref{ogr}}\label{dok-ogr}
We will derive the desired result from the following statement that, in an equivalent form, is due to Hille \cite{hille}.
\begin{prop}\label{hhille}
Let $G\in L_{1,\loc}(0,\w),\ G\ge0$. The operator $\BT_G$ is bounded if and only if
\begin{equation}\label{TGogr}
    \b_0(G):=\sup\limits_{t>0}\left(t\int_t^\w G(s)ds\right)<\infty.
\end{equation}
If this condition is fulfilled, then
\begin{equation}\label{ocenka}
    \b_0(G)\le \|\BT_G\|\le 4\b_0(G).
\end{equation}

The operator $\BT_G$ is compact if and only if
\begin{equation*}
    t\int_t^\w G(s)ds\to 0\qquad {\rm{as}}\ t+t^{-1}\to \w.
\end{equation*}
\end{prop}
A simple proof of the first statement can be found, e.g., in \cite{BS-ftprob}, Proposition 4.3. The second statement follows in a standard way.\vs

To derive from here \thmref{ogr}, suppose first that $\gz_j(G)\le K$ for all $j\in\Z$. Take any $t>0$, and let $j_0\in\Z$ be the unique number, such that $2^{j_0}<t\le 2^{j_0+1}$. Then
\begin{equation*}
\int_t^\w G(s)ds\le \int_{2^{j_0}}^\w G(s)ds
=\sum_{j> j_0}\int_{I_j}G(s)ds=
\sum_{j> j_0}2^{-j}\gz_j(G)\le 2^{-j_0}K,
\end{equation*}
whence \eqref{TGogr} is fulfilled with $\b_0(G)\le2\|\bsymb\gz(G)\|_\w.$

Let now \eqref{TGogr} be satisfied. Then
\[\gz_j(G)= 2^j\int_{I_j}G(s)ds\le \b_0(G),\]
and \eqref{TGogr1} is fulfilled with $\b_0(G)\ge K/2$. The first statement of Theorem follows from \eqref{ocenka}.
As in the case of Proposition \ref{hhille}, the second statement follows from here in a standard way, and this completes the proof.
\subsection{Proof of \thmref{estbelow}.}\label{dok-below}
The next argument is borrowed from the paper \cite{BL}, Subsection 4.3.

Let us fix an arbitrary non-negative function $f\in C_0^\w(2^{-3/2},2^{1/2})$, such that
$f(t)=1$ for $1<t<2$, and denote $u_j(t)=f(2^{-j}t),\ j\in\Z$. Note that the supports of
$u_j$ and $u_{j+2}$ do not intersect. Besides,
\[ \int_0^\w|u'_j(t)|^2dt=2^{-j}\g;\qquad \g=\int_0^\w|f'(t)|^2dt,\]
and
\[\int_0^\w G|u_j|^2dt\ge\int_{I_j}Gdt= 2^{-j}\gz_j(G).\]

If $u(t)=\sum\limits_j c_{2j} u_{2j}(t)$, then
\[ \int_0^\w|u'(t)|^2dt=\g\sum_j 2^{-2j}|c_{2j}|^2;\qquad \int_0^\w G|u|^2dt\ge\sum_j 2^{-2j}\gz_{2j}(G)|c_{2j}|^2.\]
By the variational principle this yields that
\[ n_+(s,\BT_G)\ge n_+(\g s,\{\gz_{2j}(G)\}).\]
The same inequality holds for the sequence $\{\gz_{2j-1}(G)\}$ on the right. These two inequalities immediately imply \eqref{lowestim}. The proof is complete.
\subsection{Proof of \thmref{S1/2}.}\label{dok-1/2}
Consider firstly the case of a finite interval $I=(a,b)\subset\R$. Let $\BQ_{I,G}$ be
the operator that corresponds to the Rayleigh quotient
\begin{equation}\label{rayfin}
    \frac{\int_I G|u|^2dt}{\int_I(|u'|^2+|u|^2)dt},\qquad u\in H^1(I).
\end{equation}
The eigenpairs of $\BQ_{I,G}$ solve the boundary value problem
\[ \l(-u''(t)+u(t))=G(t)u(t),\ a<t<b;\qquad u'(a)=u'(b)=0.\]
The behavior of its spectrum is well known. Estimate \eqref{est-konpr}
below is the simplest particular case (for $m=l=a=1,\Om=I,\ d\rho=Gdt$) of Theorem 4.1, Statement 2), in the book \cite{BSlect}. Besides, in Appendix we present its short proof that does not rely on the general techniques used in \cite{BSlect}.
\begin{prop}\label{konpr}
Let $G\in L_1(I),\, G\ge0$. Then there exists a constant $C=C(I)>0$, such that
\begin{equation}\label{est-konpr}
    n_+(s,\BQ_{I,G})\le Cs^{-1/2}\left(\int_I Gdt\right)^{1/2}.
\end{equation}
\end{prop}

The result survives, up to the value of the estimating constant $C(I)$, if we replace the denominator in \eqref{rayfin} by
$\int_I(|u'|^2+|u|^2t^{-2})dt$.
\vs

Consider now the Rayleigh quotient \eqref{rayfin} with the integration over the interval $I(h)=(ah,bh),\ h>0$. The substitution $t=hs,\ u(t)=v(s)$ reduces it to the form \eqref{rayfin} for the original interval $I$, with the function $G_h(s)=h^2G(hs)$. Applying the estimate \eqref{est-konpr}, we get
\begin{equation*}
    n_+(s,\BQ_{I(h),G})\le Cs^{-1/2}\left(h\int_{I(h)}Gdt\right)^{1/2},
\end{equation*}
with the constant $C=C(I)$ that does not depend on $h$. In particular, take $I=I_0=(1/2,1)$ and
$h=2^j$. Then we conclude that the estimate
\begin{equation}\label{estforIj}
    n_+(s,\BQ_{I_j,G})\le Cs^{-1/2}\gz_j^{1/2}(G)
\end{equation}
is satisfied with the constant that does not depend on $j\in\Z$.
Now, applying the estimate \eqref{estforIj} to each interval $I_j$ and using the variational principle, we obtain \eqref{basicest1}.\vs

The asymptotic formula \eqref{weyl} is well-known, say, for the potentials $G\in C_0^\w(\R_+)$.
This class is dense in the space defined by the condition $\bsymb\gz(G)\in\ell_{1/2}$, and the formula \eqref{weyl} for all such potentials follows from the general theorem on the continuity of asymptotic coefficients, see Theorem 11.5.6 in the book \cite{BSbook}.

The proof is complete.
\subsection{Proof of \thmref{interp}.}\label{proof-int} The upper estimates in \eqref{est1}, \eqref{est2}, and \eqref{est3} follow from \eqref{TGogr1} and \eqref{basicest}
by interpolation (we use the real interpolation method for the quasi-normed spaces, see \cite{BeL}).
The lower estimates immediately follow from \thmref{estbelow}.

\section{Behavior of the function $N_-(\BM_{\a G})$. Related problems}\label{bbeh}
Using the Birman -- Schwinger principle (i.e., the equality \eqref{bsch}), it is easy to reformulate the results of Section \ref{mainr} in terms of the function $N_-(\BM_{\a G})$. Recall that  the quadratic form $\bm_{\a G}$ is bounded from below and closed for all $\a>0$ at once if and only if the operator $\BT_G$ is compact. So, by \thmref{ogr} the condition $\gz_j(G)\to0$ as $j\to\pm\w$ guarantees that each operator $\BM_{\a G},\ \a>0$, is well-defined.

The results below are immediate consequences of Theorems \ref{S1/2} and \ref{interp}. We only note that the property $\BT_G\in\GS_q$ cannot be conveniently re-formulated in terms of the function $N_-(\BM_{\a G})$. For this reason, we use the inclusion $\bsymb\gz(G)\in\ell_q$ only as a sufficient condition for $\BT_G\in\Sg_q^\circ$.

\begin{thm}\label{S1/2a}
Let $G\ge 0$ be a function on $\R_+$, integrable on each interval $(a,b),\ 0<a<b<\w$. Define the sequence $\bsymb\gz(G)$ as in \eqref{posl}.

$\bsymb1$. If $\bsymb\gz(G)\in\ell_{1/2}$, then
\begin{equation*}
    N_-(\BM_{\a G})\le C\a^{1/2}\sum_{j\in\Z}\gz^{1/2}_j(G),
\end{equation*}
with a constant $C>0$ that does not depend on $G$,
and the asymptotic formula \eqref{weylalpha} holds.

$\bsymb2$. Let $q>1/2$. The function $N_-(\BM_{\a G})$ behaves as $O(\a^q)$ if and only if $\bsymb\gz(G)\in\ell_{q,\w}$, and it behaves as $o(\a^q)$ if and only if $\bsymb\gz(G)\in\ell_{q,\w}^\circ$, in particular if $\bsymb\gz(G)\in\ell_q$. The two-sided estimate is satisfied:
\begin{equation*}
    c_q\D_q(\bsymb\gz(G))\le \limsup\limits_{\a\to\w}\a^{-q}N_-(\BM_{\a G})\le C_q\D_q(\bsymb\gz(G)).
\end{equation*}
\end{thm}

\begin{example}\label{primer}
Let $G(t)=t^{-2}(\ln t)^{-1/q}$ for $t>e$ and $G(t)=0$ otherwise. Then $\gz_j(G)\sim cj^{-1/q},\ c>0$, as $j\to\w$ and hence, $\bsymb\gz(G)\in\ell_{q,\w}$ and $\D_q(\bsymb\gz(G))\neq0$. By \thmref{S1/2a}, for $q>1/2$ we have $N_-(\BM_{\a G})=O(\a^q)$,
and the estimate is sharp. This example is borrowed from the paper \cite{BL} where it was also
shown that the function $N_-(\BM_{\a G})$ has a regular asymptotic behavior of the order $\a^q$.
\end{example}

\subsection{The Neumann boundary condition at $t=0$.}\label{drugie}
Suppose that instead of the Dirichlet boundary condition $u(0)=0$ in \eqref{poluos-a}, we impose the Neumann condition $u'(0)=0$. We denote the resulting operator by $\BM_{\CN,\a G}$, and its quadratic form by $\bm_{\CN,\a G}$:
\begin{equation}\label{noimf}
    \bm_{\CN,\a G}[u]=\int_0^\w (|u'|^2-\a G(t)|u|^2)dt,\qquad u\in H^1(\R_+).
\end{equation}

 The results for the function $N_-(\BM_{\CN,\a G})$ can be easily derived from the ones for
$N_-(\BM_{\a G})$. However, an important difference appears due to the fact that now the potential $G$ must be integrable near the point $t=0$, since otherwise the quadratic form
\eqref{noimf} cannot be bounded from below. Hence, it is reasonable, instead of \eqref{posl}, to consider the "one-sided" sequence
\begin{equation}\label{posl-n}
    \bsymb\gz(\CN,G)=\{\gz_j(\CN,G)\}_{j\ge0}:\qquad \gz_0(\CN,G)=\int_0^1Gdt;\ \gz_j(\CN,G)=\gz_j(G)\ {\rm {for}}\ j>0;
\end{equation}
cf. Remark \ref{rem2}.
\begin{thm}\label{noj}
Let $G\ge0$ be a function on $\R_+$ integrable on each finite interval $(0,b)$. Define
the sequence $\bsymb\gz(\CN,G)$ as in \eqref{posl-n}. Then

$\bsymb{1.}$ If $\bsymb\gz(\CN,G)\in\ell_{1/2}$, then
\begin{equation*}
    N_-(\BM_{\CN, \a G})\le 1+ C\a^{1/2}\sum_{j\ge0}\gz_j^{1/2}(\CN,G),
\end{equation*}
and the asymptotic formula \eqref{weylalpha} holds for the operator $\BM_{\CN, \a G}$.

$\bsymb2$. The result of Theorem $\ref{S1/2a}, \bsymb2$ is satisfied for $\BM_{\CN, \a G}$, with the sequence $\bsymb\gz(G)$ replaced by $\bsymb\gz(\CN,G)$.
\end{thm}
For the proof, it is sufficient to note that $H^{1,0}(\R_+)$ is a subspace of co-dimension $1$ in $H^1(\R_+)$.
\subsection{Operator on the whole line}\label{pryam}
In quite a similar way, the results can be applied to the operator on the whole line,
\begin{equation*}
    (\BM_{\R,\a G}u)(t)=-u''(t)-\a G(t)u(t),\ t\in\R.
\end{equation*}
Indeed, by imposing the additional condition $u(0)=0$, one reduces the resulting operator to the direct orthogonal sum of two operators of the type \eqref{poluos-a}. Here it is convenient to use the sequence
\begin{equation}\label{sum}
   \wh{\bsymb \gz}(G)=\{\wh{\gz_j}(G)\}_{j\ge0}:\qquad
     \wh{\gz_0}(G)=\int_{-1}^1G(t)dt, \quad
   \wh{\gz_j}(G)=2^j\int_{2^{j-1}<|t|<2^j}G(t)dt\quad  (j\in\N).
   \end{equation}

The following result is the direct analog of \thmref{noj}.
\begin{thm}\label{prya}
 Let $G\ge0$ be a function on $\R$ integrable on each finite interval $(a,b)$. Define
the sequence $\wh{\bsymb\gz}(G)$ as in \eqref{sum}. Then

$\bsymb1$. If $\wh{\bsymb\gz}(G)\in\ell_{1/2}$, then
\begin{equation}\label{basicest1r}
    N_-(\BM_{\R, \a G})\le 1+ C\a^{1/2}\sum_{j\ge0}\wh\gz_j^{1/2}(G),
\end{equation}
and the asymptotic formula \eqref{weylalpha} $($with the integration over $\R$$)$ holds for the operator $\BM_{\R, \a G}$.

$\bsymb2$. The result of Theorem $\ref{S1/2a},  \bsymb2$, is satisfied for $\BM_{\R, \a G}$, with the sequence $\bsymb\gz(G)$ replaced by $\wh{\bsymb\gz}(G)$.
\end{thm}\vs

Among the other results, let us mention the estimate
\begin{equation}\label{mar}
    N_-(\BM_{\R, \a G})\le 1+ \sqrt{2\a}\left(\int_\R t^2G(t)dt\int_\R G(t)dt\right)^{1/4}
\end{equation}
presented in \cite{chad}, with the reference to an earlier paper \cite{Ma}.
It applies to the potentials, such that
\[ \int_\R(1+t^2)G(t)dt<\w.\]
The latter condition is quite restrictive.
Still, the presence of an explicitly given constant in (5.5) gives it a certain interest.
\vs

One more attempt to investigate this problem was recently undertaken in the preprint \cite{mv}. The authors' goal was to adapt Lieb's approach of proving CLR estimate \eqref{clr} to the low-dimensional cases. It is well known that this approach does not work well in dimensions
1 and 2, and the result of \cite{mv} for $d=1$ is much weaker than our \thmref{prya}. In particular, the authors present an example that actually coincides with our Example \ref{primer}, for $q=1$. Their result shows that $N_-(\BM_{\a G})<\w$ for all $\a>0$, but gives an excessive estimate for its growth rate.
\subsection{Comparison with the estimates in other dimensions.}\label{srav}
It is well known that for the operator $-\D-\a V$ on $\R^d$ the asymptotic formula
\begin{equation}\label{weyl-d}
    N_-(-\D-\a V)\sim c_0(d)\a^{d/2}\int_{\R^d}V^{d/2}dx,\qquad c_0(d)=v_d(2\pi)^{-d},
\end{equation}
is satisfied, under some appropriate assumptions on the potential $V\ge0$. In \eqref{weyl-d} $v_d$ stands for the volume of the unit ball in $\R^d$.

The conditions on the potential $V$, guaranteeing the validity of \eqref{weyl-d},
depend on the dimension $d$. For $d>2$ the CLR estimate \eqref{clr}
shows that the function $N_-(\D-\a V)$ is estimated through its own asymptotics.

For $d=1,2$ the situation is different. In particular, for $d=1$ a similar estimate is
impossible, since the inclusion $V\in L_{1,\loc}$ is necessary for the quadratic form
\eqref{qf} to be well-defined.

The necessary and sufficient conditions for the validity of \eqref{weyl-d} for $d=1$ can be easily derived from Theorem D in \cite{NS},
and they are much stronger than those guaranteeing \eqref{kvazi}. As it was already mentioned (in an equivalent form) in Section \ref{mainr}, in the paper \cite{NS} several examples were constructed where $N_-(\BH_{\a G})=O(\a^{1/2})$ but the asymptotic formula \eqref{weyl-d} (for $d=1$) fails.

The CLR estimate for $d=2$ fails. Instead, the opposite inequality
is satisfied,
\[ N_-(-\D-\a V)\ge c\a\int_{\R^2}Vdx.\]
It was was established in \cite{GNY}.

\section{Applications}\label{prim}
The problem discussed in the previous part of this paper can be treated as a special case of the general problem in which one studies the behavior of the negative spectrum of the Schr\"odinger operator
\begin{equation*}
    \BH_{\a V}=-\D-\a V
\end{equation*}
in a domain $\Om\in\R^d$ under some boundary condition at $\p\Om$, or on a manifold. The general effect studied here is the birth of the eigenvalues from the edge $\l_0$ of the continuous spectrum of the unperturbed operator $-\D$. Hence, those problems are of interest, where $\l_0=0$. Otherwise, one should consider $-\D-\l_0$ as the unperturbed operator. If $\Om$ is a domain, and if $d>2$, then the situation for the Dirichlet Laplacian, $\D_\CD$, is governed by the same rules as for the whole space, since the Cwikel -- Lieb -- Rozenblum estimate (CLR estimate)
\begin{equation}\label{clr}
    N_-(-\D_\CD- V)\le C(d)\int_\Om V^{d/2}dx,\qquad \Om\subseteq\R^d, \ d>2,
\end{equation}
is satisfied in any domain $\Om$, with the same constant as for the whole of $\R^d$.

In the case of the Neumann Laplacian, $\D_\CN$, the picture can be different.
Usually it happens due to some special geometric features of the domain. Similar problems arise when the unperturbed operator has a different nature, for example if it is an elliptic operator with periodic coefficients. The reader finds a profound discussion of this class of
problems in the papers \cite{Bthr1,Bthr2} by M.Sh. Birman. We, in this paper, restrict ourselves by giving a couple of typical examples.
\subsection{Perturbed Neumann Laplacian in a cylinder.}\label{cyl}
Let $\Om=\Om_0\times\R_+$ be a semi-infinite cylinder in $\R^d, \ d\ge 3$:
\begin{equation*}
    \Om=\{x=(y,t):y\in\Om_0,t>0\},
\end{equation*}
where $\Om_0$ is a bounded domain in $\R^{d-1}$ with  smooth boundary.
 We are interested in the behavior of the number $N_-(\D_\CN-\a V)$
as $\a\to\w$. The operator $\D_\CN-\a V$ is rigorously defined via its quadratic form
\begin{equation}\label{ncyl-form}
    \bh_{\CN,\a V}[u]=\int_\Om(|\nabla u|^2-\a V|u|^2)dx,\qquad u\in H^1(\Om).
\end{equation}
The form-domain contains the subspace $\CX$ of functions depending only on $t$, $u(y,t)=w(t)$. On $\CX$ we have
\[ \bh_{\CN, \a V}[u]=\bm_{\a G}[w]\]
where $\bm_G$ is as in \eqref{qf} and
\begin{equation}\label{vspom}
    G(t)=\frac1{\meas\Om_0}\int_{\Om_0}V(y,t)dy.
\end{equation}
The decomposition $H^1(\Om)=\CX\oplus\CX^\perp$ does not diagonalize the quadratic form \eqref{ncyl-form}. Still,
the operator family $\BM_{\a G}$ affects the spectrum, and its influence is reflected both in the estimates and in the asymptotic formulas for $N_-(\BH_{\CN,\a V})$.

Below we formulate the result. See \cite{S97} for its proof, and also for some other examples of a similar nature.
\begin{thm}\label{cilindr}
Let $d>2,\ \Om_0\subset\R^{d-1}$ be a bounded domain with smooth boundary, $\Om=\Om_0\times\R_+$, and let $V\in L_{1,\loc}(\Om),\ V\ge 0$. Define the function $G$ as in \eqref{vspom} and the corresponding sequence $\bsymb\gz(\CN,G)$ as in
\eqref{posl-n}. Then

{\bf 1.} $N_-(\BH_{\CN,\a V})=O(\a^{d/2})$ if and only if $V\in L_{d/2}(\Om)$ and $\bsymb\gz(G)\in\ell_{d/2,\w}$. Under these assumptions the following estimate and asymptotic formulas are valid:
\begin{equation*}
    N_-(\BH_{\CN,\a V})\le 1+C\left(\int_\Om V^{d/2}dx+\sup_{s>0}(s^{d/2}n_+(s,\bsymb\gz(G))\right);
\end{equation*}
\begin{gather*}\limsup\limits_{\a\to\w}\a^{-d/2}N_-(\BH_{\CN,\a V})=c_0(d)\int_\Om V^{d/2}dx+\D_{d/2}(\BT_G);\\
\liminf\limits_{\a\to\w}\a^{-d/2}N_-(\BH_{\CN,\a V})=c_0(d)\int_\Om V^{d/2}dx+\d_{d/2}(\BT_G),
\end{gather*}
where $c_0(d)$ is the classical Weyl coefficient.
In particular, the Weyl formula holds if and only if $V\in L_{d/2}(\Om)$ and $\bsymb\gz(G)\in\ell_{d/2,\w}^\circ$.

{\bf 2.} If $V\in L_{d/2}(\Om)$ and $\bsymb\gz(G)\in\ell_{q,\w}$ with some $q>d/2$, then
\begin{equation*}
   \limsup\limits_{\a\to\w}\a^{-q}N_-(\BH_{\CN,\a V})=\d_q(\BT_G),\qquad
    \liminf\limits_{\a\to\w}\a^{-q}N_-(\BH_{\CN,\a V})=\D_q(\BT_G).
\end{equation*}
\end{thm}

Recall that the functionals $\D_q(\BT), \d_q(\BT)$ for any $q>0$ were defined in \eqref{Dd},
\eqref{Ddop}.

The same effect manifests itself in some problems of asymptotics
in the spectral parameter, see \cite{S98,kov}.

\subsection{\sh operator on $\R^2$}\label{D2}
This was historically the first example of the problem in which the effect discussed had been revealed. This was done in the papers \cite{S-dim2} for estimates, and \cite{BL} for  asymptotics. The mechanism here is more subtle that in the previous example: an auxiliary
operator on the line appears due to the special character of the Hardy inequality in dimension $2$.

The result presented below was obtained in the paper \cite{LS}. It concerns the operator $-\D-\a V$ on $\R^2$, with the {\it radial} potential, that is, $V(x)=F(|x|)$ where $F\ge0$ is a given function on $\R_+$. The subspace in $H^1(\R^2)$ that gives rise to the auxiliary operator, is $\CX=\{u\in H^1(\R^2): u(x)=\varf(|x|)$\}. For $u\in\CX$ the quadratic form of the \sh operator with radial potential
becomes
\[ \int_{\R^2}(|\nabla u|^2-\a F(|x|)|u|^2)dx=2\pi\int_0^\w(|\varf'_r|^2-\a F(r)|\varf|^2)rdr.\]
The substitution $r=e^t,\ \varf(r)=\om(t)$ reduces it to the form
\[ 2\pi\int_\R(|\om'_t|^2-\a G_F(t)|\om|^2)dt,\qquad G_F(t)=e^{2|t|}F(e^t).\]

The following result is Theorem 5.1 in \cite{LS}.

\begin{thm}\label{2rad}
Let $d=2$ and $V(x)=F(|x|)\ge0$. Define an auxiliary one-dimensional potential
\begin{equation*}
    G_F(t)=e^{2|t|}F(e^t),\qquad  t\in\R,
\end{equation*}
and let $\wh{\gz}(G_F)$ be the corresponding sequence \eqref{sum}.
Then
$N_-(\BH_{\a V})=O(\a)$ if and only if $V\in L_1(\R^2)$ and $\wh{\gz}(G_F)\in\ell_{1,\w}$. Under these two assumptions the estimate is satisfied:
\begin{equation*}
    N_-(\BH_{\a V})\le 1+\a\left(\int_0^\infty r F(r)dr +C\|\wh{\gz}(G_F)\|_{1,\w}\right),
\end{equation*}
with some constant $C$ independent on $F$.
\end{thm}

An analogue of Statement 2 in \thmref{cilindr} is also valid but we do not present it here.\vs

In the paper \cite{LS1} an estimate was obtained for for the general (i.e., not necessarily radial) potentials. This estimate also involves the term $\|\wh{\gz}(G)\|_{1,\w}$ where
$G$ is some "effective potential" on the line. The nature of this term is the same as in
\thmref{2rad}. However, not all the difficulties, appearing in the dimension two, were overcome, and the result gives only some sufficient conditions for $N_-(\BH_{\a V})=O(\a)$ and for the validity of the Weyl type asymptotic formula.\vs

We conclude this section with the following remark. The approach described applies to many other problems, and the auxiliary operator appearing in this way not necessarily acts on the line, or on the half-line. For example, we could consider the operator $-\D_\CN-\a V$ in the domain
$\Om_0\times\R^\nu,\ \Om_0\subset\R^{d-\nu},$ with any $\nu<d$. Then the auxiliary operator would act on $\R^\nu$. The common feature of all such problems is that in order to obtain the sharp condition for the behavior $N_-(\BH_{\a V})=O(\a^{d/2})$, we need the estimates of the same order for the auxiliary operator. For a problem on $\R^\nu, \ \nu<d$, this order is {\it super-classical}. In this respect, the special peculiarity of the problems discussed in this paper is that only for $\nu=1$ the nature of such super-classical estimates is completely understood.

\section{Appendix.
Proof of Proposition \ref{konpr}}\label{app}
Let $G\in L_1(I),\ G\ge0$, where $I$ is a finite interval of
the length $l$.  For definiteness, we take $I=(0,l)$. Let
$0=t_0<t_1<\ldots<t_n=l$ be a partition of $I$ into $n$
sub-intervals $I_k=(t_{k-1},t_k)$.  Below $\Xi$ stands for
any such partition, and $|\Xi|$ stands for the number of
sub-intervals in $\Xi$, i.e., $|\Xi|=n$.
Given a real number $a>0$, consider the following
"function of partitions":
\begin{equation*}
    \Phi(\Xi)=\max\limits_k(t_k-t_{k-1})^a\int_{I_k}G(t)dt.
\end{equation*}
We need the following
\begin{lem}\label{frazb}
For any $n\in\N$ there exists a partition $\Xi$ of the
interval $I$, such that $|\Xi|=n$ and
\begin{equation}\label{oc}
    \Phi(\Xi)\le l^a n^{-1-a}\int_I G(t)dt.
\end{equation}
\end{lem}
\begin{proof}
By scaling, we conclude that it is sufficient to prove
\eqref{oc} for $l=1$ and $\int_I G(t)dt = 1$. We use
induction. For $n=1$ it is nothing to prove. Suppose
now that \eqref{oc} is satisfied for some $n$, and show
that then this is true also for $n+1$.

To this end, take a point $x$, such that
\[ (1-x)^a\int_x^1G(t)dt=(n+1)^{-1-a}.\]
Such $x$ does exist by Cauchy's theorem. Then we
have
\[\int_0^xG(t)dt=1-(n+1)^{-1-a}(1-x)^{-a}.\]
By the induction assumption, there exists a partition
$\Xi$ of $(0,x)$ into $n$ intervals, such that
\[ \Phi(\Xi)\le x^a\left(1-(n+1)^{-1-a}(1-x)^{-a}\right)
n^{-1-a}.\]
To prove Lemma, we need to show that
$\Phi(\Xi)\le (n+1)^{-1-a}$. This will be achieved if
we show that
\[ (n+1)^{-a}x^{-a}+(n+1)^{-1-a}n^{-1-a}(1-x)^{-a}
\ge 1.\]
A standard procedure shows that the function on the left attains its minimal value on $(0,1)$ at the point $x=n(n+1)^{-1}$ and this minimal
value is equal to $1$.
\end{proof}\vs

Now we are in a position to finish the proof of Proposition
\ref{konpr}.

Take any partition $\Xi$ of the interval $I$ into $n$ subintervals, and consider the subspace $\CF\subset H^1(I)$
of co-dimension $n$, formed by the functions $u$ such that
$u(t_1)=\ldots=u(t_n)=0$. For any $x\in I_k$ we have
\[ |u(t)|^2=|u(t)-u(t_k)|^2\le(t_k-t_{k-1}) \int_{I_k}|u'(s)|^2ds,\]
whence
\[ \int_{I_k}G(t)|u(t)|^2dt\le(t_k-t_{k-1})\int_{I_k}G(t)dt \int_{I_k}|u'(s)|^2ds.\]
This implies
\begin{gather*}
\int_IG(t)|u(t)|^2dt\le\sum_{k=1}^n(t_k-t_{k-1})\int_{I_k}G(t)dt \int_{I_k}|u'(s)|^2ds\\
\le\max\limits_k\left(t_k-t_{k-1})\int_{I_k}G(t)dt \right)\|u'\|^2_{L_2}.
\end{gather*}
So, we come to the situation described in \lemref{frazb}, with
$a=1$. By using Lemma and the variational principle, we conclude that
\[\l_n(\BQ_{I,G})\le l n^{-2}\int_I G(t)dt\]
for all $n>1$. Further, it follows from the embedding $H^1(I)\subset C(\overline I)$ that
\[ \l_1(\BQ_{I,G})=\|\BQ_{I,G}\|\le C\int_I G(t)dt,\]
with some constant $C=C(l)$. This is equivalent to \eqref{est-konpr}. The proof of Proposition \ref{konpr} is
complete.

\end{document}